\documentclass[11pt,a4paper,reqno]{amsart}
\usepackage{amssymb,times,mathptmx,hyperref}

\numberwithin{equation}{section}

\newtheorem{thm}{Theorem}
\numberwithin{thm}{section}
\newtheorem{prop}[thm]{Proposition}
\newtheorem{lem}[thm]{Lemma}

\theoremstyle{definition}

\theoremstyle{remark}
\newtheorem{rem}[thm]{Remark}
\providecommand{\BBb}[1]{{\mathbb{#1}}}
\providecommand{\cal}[1]{{\mathcal{#1}}}

\newcommand{\C}{{\BBb C}}
\newcommand{\Cb}{C_{\operatorname{b}}}

\newcommand{\dual}[2]{\langle\,#1,\,#2\,\rangle}

\newcommand{\fracc}[2]{{
                \textstyle\frac{#1}{\raise 1pt\hbox{$\scriptstyle #2$}}}}
\newcommand{\fracp}{\fracc1p}
\newcommand{\fracnp}{\fracc np}

\newcommand{\fracci}[2]{{\frac{#1}{\raise 1pt\hbox{$\scriptscriptstyle #2$}}}}
\newcommand{\fracpi}{\fracci1p}

\newcommand{\g}{\gamma_{\hspace{-0.17pt}0} }
\newcommand{\grad}{\operatorname{grad}}
\renewcommand{\Im}{\operatorname{Im}}
\newcommand{\mlap}{-\!\operatorname{\Delta}}

\newcommand{\norm}[2]{\mathinner{\|}#1\,|#2\|}
\newcommand{\Norm}[2]{\mathinner{\bigl\|\,#1\,\big|#2\bigr\|}}
\newcommand{\op}[1]{\operatorname{#1}}

\newcommand{\rt}{r_0}   
\renewcommand{\Re}{\operatorname{Re}}
\newcommand{\R}{{\BBb R}}
\newcommand{\Rn}{{\BBb R}^{n}}
\newcommand{\Rp}{\overline{{\BBb R}}_+}

\newcommand{\supp}{\operatorname{supp}}
\newcommand{\Z}{\BBb Z}

\title[Traces of Besov spaces]{Traces of Besov spaces revisited}
\author{Jon Johnsen}
\address{Department of Mathematical Sciences, 
Aalborg University, Fredrik Bajers Vej 7E, DK-9220 Aalborg O; Denmark}
\email{jjohnsen@math.auc.dk}
\subjclass{46E35}
\keywords{Distributional trace operator, borderline cases, mixed-norm estimate, convergence criteria, elliptic boundary
problems}
\thanks{~\\[2\baselineskip]%
{\tt Appeared in Journal of analysis and its applications (Zeischrift f{\"u}r Analysis und ihre Anwendungen),
vol.~19 (2000), no.~3, 763-779}}
\begin{document}
\begin{abstract}
For the trace of Besov spaces $B^s_{p,q}$ onto a hyperplane, the
borderline case with $s=\fracnp-(n-1)$ and $0<p<1$ is analysed and a new
dependence on the sum-exponent $q$ is found. Through examples
the restriction operator defined for $s$ down to $1/p$, and valued in
$L_p$, is shown to be distinctly different and, moreover, unsuitable for
elliptic boundary problems. 
All boundedness properties (both new and previously known) are
found to be easy consequences of a simple mixed-norm estimate, which also yields
continuity with respect to the normal coordinate. The surjectivity for
the classical borderline $s=\fracp$ ($1\le p<\infty$) is given a simpler proof
for all $q\in\,]0,1]$, using only basic functional analysis.
The new borderline results are based on corresponding convergence
criteria for series with spectral conditions.
\end{abstract}
\maketitle
\section{Introduction} \label{intro-sect}
\enlargethispage{2\baselineskip}\thispagestyle{empty}

This note concerns the (distributional) trace operator $\g$ that restricts
to the hyperplane $\Gamma:=\{x_n=0\}$ in $\Rn$ for $n\ge 2$,
\begin{equation}
  \label{bsc-op}
  \gamma_0\colon f(x_1,\dots,x_n)\mapsto f(x_1,\dots,x_{n-1},0).
\end{equation}
The title should indicate both that there remains unexplored
borderlines in the $L_p$-theory of $\g$ and that the existing litterature
do not reveal the full efficacy of the Fourier analytic proof methods.

\bigskip

The main purpose is to describe the borderline cases for $0<p<1$.
See the below Theorem~\ref{nbord-thm} concerning
$s=\fracnp-n+1$, where it is shown that the smallest Besov space containing
$\g(B^{s}_{p,q})$ has its integral-exponent equal to $\max(p,q)$, hence
depending on both the integral- and the sum-exponent of the domain. This
result seems to be hitherto undescribed.

Secondly Theorem~\ref{nbord-thm} is proved in a mere
two lines, deriving from  the
Paley--Wiener--Schwartz theorem and the
Nikolski\u\i--Plancherel--Polya inequality  a basic mixed-norm, in fact
$L_p(\R^{n-1}; L_{\infty,x_n})$,   estimate. In addition all the known
boundedness results are recovered equally easily from the same
calculation. The ensuing \emph{unified} treatment is
in contrast with  the existing litterature, which has various page-long
arguments both for the generic cases ($s>\fracp+(n-1)(\fracp-1)_+$) and the
classical borderline $s=\fracp$ ($1\le p\le\infty$).
The present paper should also be interesting for this reason.

Thirdly, another perspective on $\g$ is also gained from
the mixed-norm estimate,  for this yields (since the value $x_n=0$ har no
special significance)
that all the treated $B^{s}_{p,q}$ are contained in  $C(\R, \cal
D'(\R^{n-1}))$ and that $\g$ is a restriction of the natural trace on the
latter space. This property has not been given much
attention in the Besov space litterature (J.~Peetre's report \cite{Pee75}
seems to be the only example), although in practice $\g$ has been defined
space by space by means of a limiting procedure. Evidently this raises the question
whether $\g u$ is consistently defined when $u$ belongs to both
$C(\Rn)$ and $B^1_{1,1}(\Rn)$ or to another intersection of two
spaces. However, the consistency is always assured by the below embedding
into $C(\R,\cal D'(\R^{n-1}))$. 

Finally, the surjectiveness of 
$\g\colon B^{\fracpi}_{p,q}(\Rn) \to L_p(\R^{n-1})$ for
$1\le p<\infty$ and $0<q\le 1$ is given a new proof by an easy extension of
the Closed Range Theorem to quasi-Banach spaces.

\bigskip

For precision's sake it should be mentioned that $\g$ first of all 
refers to a \emph{working definition} of the trace
as $\g u= \sum (\check\Phi_k*u)|_{x_n=0}$, whereby $u=\sum \cal
F^{-1}(\Phi_k\hat u)$ is a Littlewood--Paley decomposition; cf.\ 
Section~\ref{definition-sect} below. Consistency and independence of the
$\Phi_k$ are obtained post-festum, cf.~\eqref{i3} and
Theorem~\ref{dist-thm} below. 
As the point of departure, the generic properties of $\g$ are recalled: 

\begin{thm}[\cite{T0},\cite{J}]
  \label{gnrc-thm}
When applied to the Besov spaces $B^s_{p,q}(\Rn)$ with $0<p,q\le\infty$, the trace
$\g$ is continuous 
\begin{equation}
  \label{i2}
  \gamma_0\colon B^s_{p,q}(\Rn)\to B^{s-\fracpi}_{p,q}(\R^{n-1})
\end{equation}
for $s>\fracp$ if $p\ge1$, and for $s>\fracnp-n+1$ if $p<1$. 
Moreover, $\g$ has a right inverse $K$,  which is bounded from
$B^{s-\fracpi}_{p,q}(\R^{n-1})$ to 
$B^{s}_{p,q}(\Rn)$ for every $s\in\R$.
\end{thm}

It is known, but proved explicitly here, that on the one hand
$\gamma_0$ in 
\eqref{i2} is a restriction of the distributional trace, that is of
\begin{equation}
  \label{i3}
  f(0)\quad \text{defined for} \quad
  f(t)\in C(\R,\cal D'(\R^{n-1})).
\end{equation}
(This is also denoted by $\g f$ in the rest of the
introduction.) 
On the other hand,
the restriction of $\gamma_0$ to the Schwartz space $\cal S(\Rn)$
extends by continuity, cf.\ \cite{Jaw78, FJ1,FJ2, T3}, to an operator  
\begin{equation}
  \label{i4}
  T\colon B^s_{p,q}(\Rn)\to L_p(\R^{n-1}),
  \quad\text{for}\quad s>\fracp, \quad 0<p<1.
\end{equation}
It should be emphasised that $T$ is rather different
from $\gamma_0$ when $s<\fracnp-n+1=\fracp+(n-1)(\fracp-1)$ (whereby $\g$
acts only on the intersection of $B^{s}_{p,q}$ and $C(\R,\cal
D'(\R^{n-1}))$, cf.~\eqref{i3}).
Their incompatibility  may be exemplified by tensorising some
$\varphi\in C^\infty_0(\R)$ equal to $1$ near $x_n=0$ with the delta
measure $\delta_0$ in $\R^{n-1}$, for 
\begin{align}
  \gamma_0(\delta_0(x')\otimes \varphi(x_n))&= \delta_0(x'),
 \label{i6} \\
  \text{whereas}\quad T(\delta_0(x')\otimes \varphi(x_n))&=0.\phantom{whereas} 
 \label{i5} 
\end{align}
Here \eqref{i6} is clear by \eqref{i3}, since $a\delta_0$ depends
continuously on the scalar $a$.

The result in \eqref{i5} is connected to the fact that the co-domain 
$L_p$ is not
continuously embedded into $\cal D'$ when $p<1$; this fact is
elementary, for when $\eta\in\cal S(\Rn)$ with $\int \eta=1$, 
then $k^n\eta(k\cdot)$ tends to $\delta_0$ in $\cal D'$ and to $0$ in $L_p$
for $k\to\infty$ because
\begin{equation}
  \norm{k^n\eta(k\cdot)}{L_p}=\norm{\eta}{L_p}k^{n(1-\fracpi)}\to 0,
  \quad\text{for each $p<1$}.
  \label{i7}
\end{equation}
With a similar $\eta\in\cal S(\R^{n-1})$ and
$\psi_k(x)=k^{n-1}\eta(kx')\varphi(x_n)$,  
\begin{align}
  \gamma_0\psi_k&= k^{n-1}\eta(kx')\to \delta_0
  \quad\text{in $\cal D'$},
  \label{i9} \\
  \text{whereas}\quad T\psi_k&= k^{n-1}\eta(kx')\to 0
  \quad\text{in $L_p$}, 
  \label{i8} 
\end{align}
so the sequence $(\psi_k)$ is treated rather differently by $\g$ and
$T$ (in fact \eqref{i5} can be proved thus, cf.\ Remark~\ref{cx1-rem}
below). 

These phenomena also depend on the \emph{domain} chosen in
\eqref{i4}. Indeed, $\g$ in \eqref{i3} is for $p<1$ continuous  
$B^{\fracci np-n+1}_{p,q}\to\cal D'$  only if $q\le 1$ (and
a fortiori not at all for $s<\fracnp-n+1$)
by \cite[Lem.~2.8]{JJ96ell}, or \cite[Lem.~2.5.2]{JJ93}; however,
the counterexample there does not contradict \eqref{i4}, cf.\
Remark~\ref{cx2-rem}. (Similarly, for $s=\fracp$ and $q>1$, hence for
$s<\fracp$, it was shown too that $\g$ is never
continuous from $B^{s}_{p,q}$, \emph{regardless} of the co-domain.)

Moreover, the severe shortcomings of $T$ in connection with elliptic boundary
problems for $s\le\fracnp-n+1$ are reviewed in Remark~\ref{ell-rem} below. 

Altogether $T$ discards so much information that it is inconsistent
with the distribution trace $\gamma_0$, seemingly to the extent that
it is inappropriate, for the usual applications, to maintain $s=\fracp$ as
the borderline when $p<1$.

\bigskip

In view of the above, it is natural to analyse $s=\fracnp-n+1$ when $p<1$.
The main point is that $q\le p \le1$ and $p<q\le1$ constitute two rather
different cases:

\begin{thm}
  \label{nbord-thm}
For $0<p<1$ the operator $\gamma_0$ is continuous 
\begin{equation}
  \gamma_0\colon B^{\fracci np-n+1}_{p,q}(\Rn)\to
B^{(n-1)(\fracpi-1)}_{p,\infty}(\R^{n-1}) \quad\text{if}\quad q\le p<1,
  \label{i10}
\end{equation}
whereas it is bounded 
\begin{equation}
  \gamma_0\colon B^{\fracci np-n+1}_{p,q}(\Rn)\to
B^{(n-1)(\fracci 1q-1)}_{q,\infty}(\R^{n-1}) \quad\text{when}\quad p<q\le1.
  \label{i11}
\end{equation}
Furthermore, $q$ is the \emph{smallest} possible
integral-exponent
for the co-domain in \eqref{i11}, for even
$B^{t}_{r,\infty}$ can only receive when $r\ge q$.
\end{thm}

This shows that the smallest Besov space one may use as
a co-domain of $\g$ is
$B^{(n-1)(\fracci1r-1)}_{r,\infty}$ with $r=\max(p,q)$
when $s=\fracnp-n+1$ and $0<p<1$; in addition neither \eqref{i10} nor
\eqref{i11} is a surjection (hence the range is not a Besov space, 
cf.\ Remark~\ref{FJS-rem} below). Altogether this makes a noteworthy
contrast with Theorem~\ref{gnrc-thm}.

To  elucidate Theorem~\ref{nbord-thm}, one can observe that the
above-mentioned operator 
$T$ is a continuous surjection, see \cite[Th.~5.1]{FJ1}, \cite[4.4.3]{T3},
\begin{equation}
  T\colon B^{\fracpi}_{p,q}(\Rn)\to L_p(\R^{n-1})
  \quad\text{for}\quad 0<q\le p<1.
\end{equation}
Here the condition $q\le p$ is known to be necessary, and formally a
distinction between the same cases appear  in Theorem~\ref{nbord-thm}
too. This seems surprising and unnoticed hitherto, and a 
fortiori the theorem is a novelty; cf.~Remark~\ref{FJS-rem} below.

As an interpretation of \eqref{i11}, note that it follows from \eqref{i10}
when combined with a Sobolev embedding. In fact, given \eqref{i10} then
\begin{equation}
  B^{\fracci np-n+1}_{p,q}(\Rn)\hookrightarrow
  B^{\fracci nq-n+1}_{q,q}(\Rn)\xrightarrow{\;\g\;}
  B^{(n-1)(\fracci1q-1)}_{q,\infty}(\R^{n-1}),
  \label{3.5}
\end{equation}
and since $q$ is the optimal integral-exponent on the right hand
side of \eqref{i11}, cf.\ Section~\ref{nbord-sect} below, this
is the \emph{only} way to apply $\g$ when $p<q\le1$.

Moreover, in both \eqref{i10} and \eqref{i11} one can take
$L_1(\R^{n-1})$ as the receiving space, for by a Sobolev embedding
into $B^1_{1,1}(\Rn)$ the question is reduced to a case (viz.~$p=1$) of the following

\begin{thm}
  \label{obord-thm}
Let $1\le p\le\infty$ and $0<q\le1$. Then $\g$ in \eqref{i3} is
bounded
\begin{equation}
  \g\colon B^{\fracpi}_{p,q}(\Rn)\to L_p(\R^{n-1}).
  \label{2.7}
\end{equation}
Moreover, \eqref{2.7} is a {\bf surjection\/} if $1\le p<\infty$ and $0<q\le 1$.
\end{thm} 
Earlier Burenkov, Gol'dman and Peetre
\cite{Pee75, BuGo79, Gol79} proved surjectiveness for $q=1$ (the
latter two even for anisotropic spaces), but the
first to consider this borderline were seemingly Agmon and
H{\"o}rmander \cite{AgH76} (cf.\ their note), who covered $p=2$. However, the
borderline itself was found in 1951 by Nikolski\u\i~\cite{Nik51}. 
Using atomic decompositions, Frazier and Jawerth \cite{FJ1} proved the
surjectivity for $0<q\le 1$. An alternative argument is 
given below by means of a short application of the Closed Range Theorem (extended
to quasi-Banach spaces); it should be interesting because of the simplicity.

Theorems \ref{gnrc-thm}, \ref{nbord-thm} and \ref{obord-thm} are
proved and re-proved here, for 
they may actually all be obtained by combining general principles with a single,
mixed-norm estimate; in its turn, this estimate follows 
straightforwardly from the Paley--Wiener--Schwartz theorem and the
Nikolski\u\i--Plancherel--Polya inequality; see
Section~\ref{boundedness-sect} below. Besides being a unified proof, it is
also simple compared to those in e.g.~\cite{BL76,T2,FJ1}.

The mixed-norm estimate actually shows $\cal S'$-convergence of
the series used as the working definition of $\g u$ in \eqref{1.1} below.
In Theorem~\ref{obord-thm} this is a
consequence of $L_p$'s completeness, and for the generic cases in
Theorem~\ref{gnrc-thm} it follows from the known convergence criteria for
series with spectral conditions, summed up in (ii) of 
Theorem~\ref{Y-thm} below.

Furthermore, a small reflection about this estimate yields

\begin{thm}
  \label{dist-thm}
Let $s\ge \fracp +(n-1)(\fracp-1)_+$, and suppose $q\le 1$ holds in the case
of equality. Then there is an inclusion
\begin{equation}
  B^s_{p,q}(\Rn)\subset C(\R,\cal D'(\R^{n-1})),
\end{equation}
and the working definition of $\g$ amounts to a restriction of the natural
trace  on $ C(\R,\cal D'(\R^{n-1}))$.
\end{thm}

For the two cases in Theorem~\ref{nbord-thm} it is also noteworthy that 
they stem from an analogous destinction in
(iii) of Theorem~\ref{Y-thm} below.
However, part (iii) of the latter theorem is actually a generalisation
of the criteria to the borderline $s=\fracnp-n$, and the necessity of
the splitting into two cases is shown in Proposition~\ref{qp-prop}.
Hence this paper also contributes to the convergence
criteria in general Besov spaces.  

\begin{rem}
  \label{FJS-rem}
 In a subsequent joint work \cite{FaJoSi98}, inspired by the present article,
 especially Theorems~\ref{nbord-thm} and \ref{dist-thm}, the traces of all
admissible Besov and Triebel--Lizorkin spaces were determined. In particular 
 the exact ranges in \eqref{i10} and \eqref{i11} were found to be
 the approximation space $A^{(n-1)(\fracpi-1)}_{p,q}$ in both cases.
 So although $r=\max(p,q)$ is the smallest possible
 integral-exponent when the co-domain is
 stipulated to be a Besov space (as in Theorem~\ref{nbord-thm} ff.\
 and throughout this paper),
 the situation is different if the scale of $A^s_{p,q}$ spaces is adopted.
\end{rem}

\subsection*{Acknowledgement} In the early stages I benefitted from
discussions with prof.\ H.~Triebel, who also kindly provided \cite{Pee75}.

\section{Preliminaries} For the general notions in distribution theory standard
notation is used, similarly to \cite{H}; $C(\R,X)$ denotes the vector 
space of continuous functions from $\R$ to $X$, and if $X$ is a Banach space,
$\Cb(\R,X)$ stands for the sup-normed space of continuous bounded functions. 

For the Besov spaces $B^{s}_{p,q}$ the conventions of \cite{Y1} are
adopted, so the norm is defined from a
Littlewood--Paley decomposition
$1=\sum_{j=0}^\infty \Phi_j(\xi)$, where the $\Phi_j(\xi)$ vanish unless
$\frac{11}{20}2^j\le|\xi|\le\frac{13}{10}2^j$ when $j>0$. This may, moreover,
be obtained by letting $\Phi_0=\Psi_0$ and $\Phi_j=\Psi_j-\Psi_{j-1}$ when
$\Psi_j(\xi)=\Psi(2^{-j}|\xi|)$ for some real $C^\infty$ function $\Psi(t)$ on
$\R$ vanishing for $t>13/10$ and equalling $1$ for $t<11/10$; in this case
$\Psi_j=\Phi_0+\dots+\Phi_j$.  

Then $B^{s}_{p,q}$ is defined to consist of the $u\in\cal S'(\Rn)$ for which
\begin{equation}
  \norm{u}{B^{s}_{p,q}}:=(\sum_{k=0}^\infty2^{skq}\norm{\cal
  F^{-1}(\Phi_k\hat u)}{L_p}^q)^{\fracci1q}<\infty.
\end{equation}
On $\R^{n-1}$ a partition of unity $1=\sum \Phi_j'$ with
$\Phi_j'(\xi')=\Phi_j(\xi',0)$ is used. 

Equivalently a partition may be used in which each function is a
product of $n$ factors, each depending on a single coordinate $\xi_j$ of $\xi$.
This is folklore, but for precision the following easy construction
and Lemma~\ref{Liz-lem} below are given. Let $\Phi^{(1)}_k$ and
$\Psi^{(1)}_k$ denote the functions obtained in the manner above for $n=1$. Then
\begin{equation}
  \tilde\Psi_{k}(\xi):= \Psi^{(1)}_k(\xi_1)\dots\Psi^{(1)}_k(\xi_n)
\end{equation}
equals $1$ in $B_{\infty}(0,\tfrac{11}{10}2^k)$, the max-norm ball of radius
$\tfrac{11}{10}2^k$, centred at the origin; 
$\supp\tilde\Psi_{k}$ lies in $B_{\infty}(0,\tfrac{13}{10}2^k)$. Now insertion of
$\Psi^{(1)}_k=\Psi^{(1)}_{k-1}+\Phi^{(1)}_k$ gives, for $k\ge1$,
\begin{gather}
  \tilde\Psi_k(\xi)= \tilde\Psi_{k-1}(\xi)+
  \sum_{\emptyset\neq J\subset\{1,\dots,n\}} \Theta_{J,k}(\xi),
\\
  \text{whereby}\quad \Theta_{J,k}(\xi)=\prod_{j\in J}\Phi^{(1)}_{k}(\xi_j)
                                        \prod_{j\notin J}\Psi^{(1)}_{k-1}(\xi_j).
\end{gather}
Letting $\Theta_{J,0}=\tilde\Psi_{0}$, this yields a smooth partition of
unity since for $\xi\in\Rn$,
\begin{equation}
1=\sum_{k=0}^\infty\sum_{J}\Theta_{J,k}(\xi).  
\end{equation}
When $k\ge1$, then evidently 
\begin{equation}
  \supp \Theta_{J,k}\subset 
  B_{\infty}(0,\tfrac{13}{10}2^k)\setminus B_{\infty}(0,\tfrac{11}{10}2^{k-1}).
\end{equation}  
Observe also the tensor product structure of the function $\Theta_{J,k}$ and that 
$\Theta_{J,k}(\xi)=\Theta_{J,1}(2^{-(k-1)}\xi)$ for $k\ge1$.

Finally, the next lemma may be proved in the usual way by means of (iv) in
Theorem~\ref{Y-thm} below, using also that independently of $k$ there are (1
or) $2^n-1$ terms in the sum over $J$.

\begin{lem}
  \label{Liz-lem}
For every $s\in\R$ and $p$, $q\in\,]0,\infty]$ the Besov space
$B^{s}_{p,q}(\Rn)$ coincides with the set of $u\in\cal S'(\Rn)$ for which the
following quasi-norm is finite:
\begin{equation}
  \norm{u}{B^{s}_{p,q}}^{\Theta}:=(\sum_{k=0}^\infty\sum_{J}2^{skq}\norm{\cal
  F^{-1}(\Theta_{J,k}\hat u)}{L_p}^q)^{\fracci1q}.
  \label{Liz-eq}
\end{equation}
Moreover, $\norm{\cdot}{B^{s}_{p,q}}^{\Theta}$ is an equivalent quasi-norm for $B^{s}_{p,q}$.
\end{lem}

\section{Definition of the trace}   \label{definition-sect}
\subsection{The working definition}
When dealing with $\gamma_0u$ it is convenient to  take a
Littlewood--Paley partition of unity, say $1=\sum_{j=0}^\infty
\Phi_j$, and let
\begin{equation}
  \gamma_0u=\sum_{j=0}^\infty \cal F^{-1}(\Phi_j\cal F u)\big|_{x_n=0}
  \label{1.1}
\end{equation}
for those $u\in\cal S'(\Rn)$ for which the sum converges in $\cal
D'(\R^{n-1})$: by the Paley--Wiener--Schwartz theorem each summand
$\cal F^{-1}(\Phi_j\cal F u)$ is an entire analytic function for which 
restriction to $x_n=0$ makes sense.

However, the limit in \eqref{1.1} might depend on the $\Phi_j$, 
but in Proposition~\ref{reconstruction-prop} below, this is shown not to be
the case for the spaces treated here.  (The procedure in \eqref{1.1} was
used to define the trace in \cite{J}, but without justification or
relation to other trace notions.)

\bigskip

The usefulness of \eqref{1.1} depends on the availability of
easy-to-apply results for the convergence of a series
$\sum_{j=0}^\infty u_j$. While for a general Banach space $X$ a finite norm
series, $\sum_{j=0}^\infty \norm{u_j}{X}<\infty$, is such a criterion,  
$B^{s}_{p,q}$ has a variant with $\ell^s_q(L_p)$-norms without the
troublesome $\cal F^{-1}\Phi_j\cal F$ acting on $u_j$.

For the reader's sake, these criteria for series with spectral conditions are
recalled with \cite[Thms.~3.6, 3.7]{Y1} in (ii) and (iv) below, 
together with supplements on the borderline cases for
$s=\max(0,\fracnp-n)$ in (i) and (iii).

\begin{thm}   \label{Y-thm}
Let a series $\sum_{j=0}^\infty u_j$ of
distributions $u_j$ in $\cal S'(\Rn)$ be given together with  numbers 
$s\in\R$ and $p$ and $q$ in $\,]0,\infty]$, and consider then
\begin{equation}
  B:=(\sum_{j=0}^{\infty} 2^{sjq}\norm{u_j}{L_p}^q)^\fracci1q
  \label{B-eq}
\end{equation}
as a constant in $[0,\infty]$ (with sup-norm over $j$ if $q=\infty$).

Then the following assertion is valid:
\begin{itemize}
    \item[(i)] If $s=0$, $1\le p\le\infty$ and $q\le1$, then
$B<\infty$ implies that $\sum u_j$ converges in $L_p(\Rn)$ to a sum
$u(x)$ for which $\norm{u}{L_p}\le B$ holds.
\end{itemize}
In addition, suppose that for some $A>0$ the spectral condition
\begin{equation}
  \supp \cal F u_j \subset \{\,\xi\mid |\xi|\le A2^j\,\}
  \label{spec-cnd}
\end{equation}
is satisfied by each $u_j$, $j\ge0$. Then one has:
\begin{itemize}
    \item[(ii)] If $s>\max(0,\fracnp-n)$, then $B<\infty$ implies
convergence of $\sum u_j$ in $\cal S'(\Rn)$ to a limit $u(x)$ in
$B^{s}_{p,q}(\Rn)$ for which $\norm{u}{B^{s}_{p,q}}\le cB$ holds for
some constant $c=c(n,s,p,q)$.
    \item[(iii)] If $s=\fracnp-n$, $p\in\,]0,1[$ and $q\in\,]0,1]$, 
then $B<\infty$ implies convergence of $\sum u_j$ in $L_1(\Rn)$ to a
limit $u(x)$ in $L_1$ for which $\norm{u}{L_1}\le cB$ holds 
for some constant $c=c(n,p,q)$.

Moreover, there is then a constant $c=c(n,p,q)$ such that
$u(x)$ belongs to $B^{\fracci np-n}_{p,\infty}$ or
$B^{\fracci nq-n}_{q,\infty}$ and satisfies the estimate 
\begin{align}
  \norm{u}{B^{\fracci np-n}_{p,\infty}}&\le cB,  
  \quad\text{when $q\le p<1$},
\\
  \norm{u}{B^{\fracci nq-n}_{q,\infty}}&\le cB,  
  \quad\text{when $p<q\le1$},
\end{align}
respectively.
  \item[(iv)] Furthermore, if the stronger condition
\begin{equation}
  \supp \cal F u_j \subset \{\,\xi\mid \frac1A 2^j\le|\xi|\le A2^j\,\}
  \label{spec-cnd'}
\end{equation}
holds for $j>0$, then the assertion \upn{(ii)} holds even for all $s\in\R$.
\end{itemize}
\end{thm} 
\begin{proof}
The completeness of $L_p$ easily gives (i); cf.\
\cite[Prop.~2.5]{JJ94mlt}. The $L_1$-part of (iii) may be reduced
to (i) by means of the Nikolski\u\i--Plancherel--Polya inequality,
cf.~\cite[Prop.~2.6]{JJ94mlt}
(modulo typos there: $L^p$ should have been $L_1$ and the corresponding
estimate `$\norm{u}{L_1}\le cB$').  

This gives the existence of $u$, and since
$\cal F^{-1}(\Phi_j\hat u)=\sum_{k=j-h}^{\infty} \cal
F^{-1}(\Phi_j\hat u_k)$ for some fixed $h\in\Z$, we may for $q\le p$
use $\ell_q\hookrightarrow \ell_p$ to get that 
\begin{equation}
  \begin{split}
  \norm{\cal F^{-1}(\Phi_j\hat u)}{L_p}&\le (\sum_{k=j-h}^\infty
  \norm{\check\Phi_j*u_k}{L_p}^p)^\fracpi
\\
  &\le c(\sum_{k=j-h}^\infty
  2^{k(\fracci np-n)p}\norm{\check\Phi_j}{L_p}^p\norm{u_k}{L_p}^p)^\fracpi
\\
  &\le c\max(\norm{\check\Phi_0}{L_p},\norm{\check\Phi_1}{L_p})
    2^{j(n-\fracci np)} B.
  \end{split}
\end{equation}
Therefore $u$ is in $B^{s}_{p,\infty}$ for $s=\fracnp-n$ with the required
estimate. For $p<q$ the Nikolski\u\i--Plancherel--Polya inequality      
applied to $B$ reduces the question to the case with $p=q$.
\end{proof}

It was also shown in \cite[Ex.~2.4]{JJ94mlt} that in both (i) and (iii) the
restriction $q\le1$ is optimal; for $q>1$ there exists series
diverging in $\cal D'(\Rn)$ for which the associated $B$ is finite. 

In addition to this, the receiving spaces in (iii)  must have sum-exponents
equal to  infinity  (see \cite[Th.~6]{FaJoSi98}, where this is derived from
trace estimates) and the
integral-exponents cannot be smaller than $p$ and $q$, respectively:

\begin{prop}
  \label{qp-prop}
If for some $t\in\R$ and $r>0$ there exists $c\in\,]0,\infty[$ such that
every $u\in\cal S(\Rn)$ satisfies 
\begin{equation}
  \norm{u}{B^t_{r,\infty}} \le c 
  (\sum_{j=0}^{\infty} 2^{j(\fracci np-n)q}\norm{u_j}{L_p}^q)^\fracci1q,
  \label{F0-cnd}
\end{equation}
whenever $u=\sum u_j$ is a decomposition satisfying
\eqref{spec-cnd}, then $r\ge q$. 

Consequently, for $p<q\le1$ in Theorem~\ref{Y-thm} \upn{(iii)},
the receiving space $B^{\fracci nq-n}_{q,\infty}$ is optimal with
respect to the integral-exponent. 
\end{prop}
\begin{proof}
The latter statement follows from the former, for on the one hand
$B^{\fracci nq-n}_{q,\infty}\hookrightarrow B^{\fracci
nr-n}_{r,\infty}$ for $r\ge q$, and if, on the other hand, 
$B^{\fracci nr-n}_{r,\infty}$ receives with an estimate
for some $r<q$, then \eqref{F0-cnd} holds. In particular it
does so when $u=\sum u_j$ is a decomposition of a
Schwartz function, so the contradicting conclusion $r\ge q$ follows.

When \eqref{F0-cnd} holds, one may for arbitrary fixed points
$x_j\in\Rn$ define 
\begin{equation}
  \omega_N=\sum_{k=1}^N \check\Psi_k(x-x_k).
\end{equation} 
Independently of the choice of the points $x_j$, the right hand side 
of \eqref{F0-cnd} equals $c N^{\fracci 1q} \norm{\check\Psi_0}{L_p}$,
and it is well known that $x_1$, $x_2$, \dots\ may be chosen such that 
\begin{equation}
  \norm{\omega_N}{B^t_{r,\infty}}\ge c(r)\cdot N^{\fracci 1r};
  \label{1r-eq}
\end{equation}
so in view of \eqref{F0-cnd} the inequality $r\ge q$ must hold.

For completeness' sake it is remarked that \eqref{1r-eq} may be seen thus:
clearly the fact that $\Psi_k\equiv 1$ on $\supp \Phi_0$ yields that 
\begin{equation}
  \norm{\omega_N}{B^t_{r,\infty}}\ge \norm{\check\Phi_0*\omega_N}{L_r}
  = \norm{\sum_{k=1}^N \check\Phi_0(\cdot-x_k)}{L_r}.
\end{equation}
Moreover, $\Phi_0(\xi)=\Phi_0(-\xi)\ge0$, so 
$\check\Phi_0$ is real-valued with $\check\Phi_0(0)>0$,
hence some $\delta>0$  fulfills that
$ \check\Phi_0(x)> \check\Phi_0(0)/2>0$ for $|x|<\delta$.

There is also $R>\delta$ such that 
$|\check\Phi_0(x)|< \check\Phi_0(0)/(2N)$ for $|x|>R$, so if
$x_j=3jR(1,0,\dots,0)$, 
\begin{equation}
  \norm{\omega_N}{B^t_{r,\infty}}\ge \tfrac{1}{2}
  (\sum_{k=1}^N \int_{B(x_k,\delta)}|\check\Phi_0(x-x_k)|^r\,dx)^{\fracci1r}
  =c(r,\Phi_0,\delta)N^{\fracci1r}.
\end{equation}
Indeed,
$|\check\Phi_0*\omega_N|\ge   \check\Phi_0(\cdot-x_j)-
\tfrac{N-1}{2N}\check\Phi_0(0)\ge  \check\Phi_0(\cdot-x_j)/2$ 
holds on the ball $B(x_j,\delta)$ because $|x_k-x|>R$ does so for $k\neq j$. 
This shows \eqref{1r-eq}.
\end{proof}

\begin{rem}
  \label{Y-rem}
In (ii) and (iv) of Theorem~\ref{Y-thm}, the series $u=\sum u_j$ converges in 
$B^{s}_{p,q}$ if $q<\infty$ and in $B^{s-\varepsilon}_{p,1}$ for
$\varepsilon>0$ if $q=\infty$. This is a well-known easy consequence of
the completeness and the norm estimate in the theorem. 
\end{rem}
\begin{rem}
  \label{spec-rem}
The spectral conditions in \eqref{spec-cnd} are robust under restriction: when
$x=(x',x'')$ is a splitting of the variables and $x''$ is kept fixed,
then
\begin{equation}
  \supp \cal F_{x'\to\xi'}u_j(\cdot,x'')\subset
  \{\,\xi'\mid |\xi'|\le A2^j\,\}
  \label{spec-cnd''}
\end{equation}
holds by the Paley--Wiener--Schwartz theorem, for $u_j(\cdot,x'')$ is still
an analytic function satisfying the relevant estimates in $\Re z'$ and 
$\Im z'$.

By the same argument, \eqref{spec-cnd'} goes over into \eqref{spec-cnd''} for
$u_j(\cdot,x'')$. 
\end{rem}

\subsection{The distribution trace}   \label{dtr-ssect}
A rather general definition of the trace is obtained as 
$\rt f:=f(0)$ on the subspace
\begin{equation}
  C(\R,\cal D'(\R^{n-1}))\subset \cal D'(\Rn).
  \label{1.2}
\end{equation}
For the spaces considered in this note, the working definition in \eqref{1.1}
actually amounts to a restriction of $\rt$. This 
is proved in Proposition~\ref{reconstruction-prop} below by means of the
injection in \eqref{1.2}, so this folklore is explicated 
(in lack of a reference):

\begin{prop}
  \label{inj-prop}
Let $f(t)$ belong to $C(\R,\cal D'(\R^{n-1}))$, whereby $\cal
D'(\R^{n-1})$ has the w$^*$-topology. Then
\begin{equation}
  \dual{\Lambda_f}{\varphi}:=\int_\R \dual{f(t)}{\varphi(\cdot,t)}\,dt,
\quad\text{for}\quad \varphi\in C^\infty_0(\Rn),
  \label{1.3}
\end{equation}
defines an injection of $C(\R,\cal D'(\R^{n-1}))$ into $\cal D'(\Rn)$.
\end{prop}
\begin{proof}
When $\varphi\in C^\infty_0$ is supported by the rectangle
$K:=[-k,k]^n$, bilinearity and the Banach--Steinhaus theorem for
$C^\infty_0([-k,k]^{n-1})$ give continuity of the map $t\mapsto
\dual{f(t)}{\varphi(\cdot,t)}$ and the bound
\begin{equation}
  |\int_{-k}^k \dual{f(t)}{\varphi(\cdot,t)}\,dt|
  \le 2k c_k \norm{\varphi}{C^\infty_K, N_k},
  \label{1.5}
\end{equation}
while $\varphi$ of the form $\psi(x')\chi(t)$ yields the injectivity
of $f\mapsto \Lambda_f$.
\end{proof}

While  it is meaningful, for every subspace $X$ of $\cal D'(\Rn)$, to ask
whether 
\begin{equation}
  X\subset C(\R,\cal D'(\R^{n-1})),
  \label{1.13}
\end{equation}
it is for arbitrary $u\in\cal D'(\Rn)$ meaningless to
ask whether the dependence on $x_n$ is continuous. Despite this peculiarity,
the estimates yielding boundedness of $\g$ in \eqref{1.1} do also give 
inclusions like \eqref{1.13} for the domains of
$\g$; cf.\ Proposition~\ref{reconstruction-prop}.

\begin{rem}
  \label{smth-rem}
On $X=\Cb(\Rn)$, where the inclusion in \eqref{1.13} is
clear, it follows that \eqref{1.1} converges to the continuous
function obtained from the operation in \eqref{bsc-op} as expected. 
Indeed, since $\Psi_k=\Phi_0+\dots+\Phi_k$ gives an approximative identity,
viz.\ $\cal F^{-1}\Psi_k$, for the convolution algebra $X$,
\begin{equation}
  u(0)=\lim_{k\to\infty}\check\Psi_k*u(\cdot,0)=\g u.
  \label{1.15}
\end{equation}
\end{rem}

\begin{rem}
  \label{core-rem}
Considering $\rho_0\colon H^1(\Rp)\to \C$  given by
$\rho_0u=u(0)$, the restriction $\rho_0\big|_{C^\infty_0}$ extends by
continuity to the zero-operator $L_2\to \C$.
This exemplifies that when a restriction of an operator is extended by
continuity between \emph{another} pair of spaces,  the resulting
map may be very different from the original one.

A less obvious example is $\g\big|_{\cal  S}$ extended as $T$
in \eqref{i4}; cf.\ \eqref{i5}--\eqref{i9}.
\end{rem}

\begin{rem}
  \label{approach-rem}
To avoid phenomena as those in Remark~\ref{core-rem}, the
approach of this paper is first of all to define $\rt$ as the
distributional trace on $C(\R,\cal D'(\R^{n-1}))$; for this reason
Proposition~\ref{inj-prop} is included.
Secondly, boundedness of $\gamma_0\colon X\to Y$ is obtained
together with the identity $\gamma_0=\rt\big|_X$ without extension by
continuity. 
\end{rem}

\section{Boundedness}
  \label{boundedness-sect}
To obtain the continuity properties, observe that since $\cal
F^{-1}(\Phi_j \hat u)$ has spectrum in the ball $B(0,R2^j)$ for
$R=\tfrac{13}{10}$, it follows from Remark~\ref{spec-rem} by freezing $x'$ that
$\cal F^{-1}(\Phi_j \hat u)(x',\cdot)$ has spectrum in $[-R2^j,R2^j]$, hence
by the Nikol'ski\u\i--Plancherel--Polya inequality that  
\begin{equation}
  \norm{\cal F^{-1}(\Phi_j \hat u)(x',\cdot)
        }{L_\infty(\R)}\le c(R2^j)^{\fracpi}
  \norm{\cal F^{-1}(\Phi_j \hat u)(x',\cdot)}{L_p(\R)},
  \label{2.6}
\end{equation}
when the latter is applied in the $x_n$-variable only.

Integration with respect to $x'$ then gives the basic $L_p$-$L_\infty$ estimate
\begin{equation}
  \Norm{\sup_{x_n\in\R}|\cal F^{-1}(\Phi_j \hat u)(\cdot,x_n)|}{L_p(\R^{n-1})} 
  \le c2^{\fracci jp}\norm{\cal F^{-1}(\Phi_j \hat u)}{L_p(\Rn)},
  \label{2.4}   
\end{equation}
and taking in particular $x_n=0$, 
\begin{equation}
  \Norm{\cal F^{-1}(\Phi_j \hat u)(\cdot,0)}{L_p(\R^{n-1})} 
  \le c2^{\fracci jp}\norm{\cal F^{-1}(\Phi_j \hat u)}{L_p(\Rn)}.
  \label{2.5}   
\end{equation}
The boundedness in Theorems~\ref{gnrc-thm},
\ref{nbord-thm} and \ref{obord-thm} now follows by Theorem~\ref{Y-thm} and
Remark~\ref{spec-rem}. 

For example, that $u\in B^{\fracpi}_{p,1}(\Rn)$ means that the right hand side
of \eqref{2.5} is in $\ell_1$, so $\sum_{j=0}^\infty
\cal F^{-1}(\Phi_j\cal F u)\big|_{x_n=0}$ converges in $L_p$
(because of its convergent norm series); hence also in $\cal D'(\R^{n-1})$
when $1\le p\le\infty$. So, with the limit denoted $\g u$
according to the working definition of $\g$,
\begin{equation}
  \norm{\g u}{L_p}\le \sum_{j=0}^\infty \norm{\cal F^{-1}(\Phi_j \hat u)(\cdot,0)}{L_p}
  \le c \norm{u}{B^{\fracpi}_{p,1}}.
\end{equation}
For $B^{\fracpi}_{p,q}$ with $0<q<1$ part (i) of Theorem~\ref{Y-thm}
applies. 

When $s=\fracnp-n+1$ for $p<1$, then \eqref{2.5} may be multiplied by
$2^{j(s-\fracpi)}$ and 
the $\ell_q$-norm of both sides calculated. By Remark~\ref{spec-rem}\,---\,this
time applied with the freezing $x_n=0$\,---\,and (iii) of Theorem~\ref{Y-thm}, the
properties in \eqref{i10}--\eqref{i11} are obtained. Observe here that the assumption on $s$ is
equivalent to  
\begin{equation}
  s-\fracp =(n-1)(\fracp-1),
\end{equation}
which is required when (iii) is applied to the co-domain
$B^{s-\fracpi}_{p,q}(\R^{n-1})$. 

In the same way \eqref{2.5} and (ii) of Theorem~\ref{Y-thm} show
the boundedness in Theorem~\ref{gnrc-thm}. 

\bigskip

Following \cite[2.7.2]{T2}, the right inverse $K$ of $\g$ may be taken as 
\begin{equation}
  Kv=\sum_{j=0}^\infty \psi(2^jx_n)\cal F^{-1}(\Phi_j' \hat v)(x')
  \label{K-eq}
\end{equation}
when $\psi\in\cal S(\R)$ has $\supp \cal F \psi\subset [-1,1]$ and
$\psi(0)=1$.

Indeed, letting $v_j=\cal F^{-1}\Phi_j'\cal Fv$,
\begin{gather}
  \supp \cal F(\psi(2^j\cdot)v_j)\subset \{\,\xi\in\Rn\mid
   2^j\le|\xi|\le3\cdot2^j\,\}   \\
  \norm{\psi(2^j\cdot)v_j}{L_p(\Rn)}= 2^{-\fracci jp}\norm{\psi}{L_p(\R)}
   \norm{v_j}{L_p(\R^{n-1})},
\end{gather}
so part (iv) of Theorem~\ref{Y-thm} gives that $Kv$ is well defined with 
\begin{equation}
  \norm{Kv}{B^{s}_{p,q}(\Rn)}\le c \norm{v}{B^{s-\fracpi}_{p,q}(\R^{n-1})}
\end{equation} 
for $s\in\R$. Moreover, for $s>\fracp +(n-1)(\fracp-1)_+$ the already shown
continuity of $\g$ gives 
\begin{equation}
  \g Kv=\sum \g(\psi(2^jx_n)v_j(x'))= \sum \psi(0)v_j = v.
\end{equation}
This reproves the claims on $K$ in Theorem~\ref{gnrc-thm}.

\begin{rem}
  \label{embd-rem}
The spaces $B^{\fracpi}_{p,1}(\Rn)$ with $1\le p\le\infty$ are maximal among
those under consideration, for when
$s>\fracp+(n-1)(\fracp-1)_+$, 
\begin{equation}
  B^{s}_{p,q}\hookrightarrow B^{\fracci1r}_{r,1} \quad\text{for $r=\max(1,p)$} 
\end{equation}
and this also holds when $s=\fracp+(n-1)(\fracp-1)_+$ and $q\le1$.
\end{rem}

\section{Continuity in the $x_n$-variable}
In view of Remark~\ref{embd-rem}, the proof of Theorem~\ref{dist-thm} need
only be conducted for the 
$B^{\fracpi}_{p,1}(\Rn)$ spaces with $1\le p\le\infty$.
Clearly $x_n=0$ does not play a special role, for the mixed norm estimate
in \eqref{2.4} `absorbs' any value equally well: obviously 
\begin{equation}
  \sup_{x_n\in\R}  \norm{\cal F^{-1}(\Phi_j \hat u)(\cdot,x_n)}{L_p(\R^{n-1})} 
  \le c2^{\fracci jp}\norm{\cal F^{-1}(\Phi_j \hat u)}{L_p(\Rn)}
  \label{2.4'}   
\end{equation}
follows in the same way as \eqref{2.5}. This means that the function series
\begin{equation}
  t\mapsto \sum_{j=0}^\infty \cal F^{-1}(\Phi_j\hat u)\big|_{x_n=t}
  \label{fu-eq}
\end{equation}
converges in the Banach space $\Cb(\R,L_p(\R^{n-1}))$, say, with the limit
denoted by $f_u(t)$. So for every 
$u\in B^{\fracpi}_{p,1}(\Rn)$ with $1\le p\le\infty$,
\begin{equation}
  f_u(t)\in \Cb(\R,L_p(\R^{n-1}))\hookrightarrow \cal D'(\Rn)
  \label{Cb-eq}
\end{equation} 
and $f_u(0)=\g u$ by the working definition of $\g$.
By \eqref{1.3}, the injection in \eqref{Cb-eq}
is well defined and continuous; in fact 
\begin{equation}
  \begin{split}
  |\dual{f}{\varphi}| &\le \int\norm{f(t)}{L_p}
         \norm{\varphi(\cdot,t)}{L_{p'}}\,dt
\\
  &\le (\op{diam}\supp \varphi)^{\fracpi}\norm{f}{\Cb(\R,L_p)} 
   \norm{\varphi}{L_{p'}(\Rn)}
  \end{split}
  \label{Cb-est}
\end{equation}
for every test function $\varphi$, when $p+p'=pp'$.
However, since the series of $C^\infty$ functions in \eqref{fu-eq}  converges
to the given $u$ in $\cal S'(\Rn)$, hence in $\cal D'(\Rn)$, it follows from
\eqref{Cb-eq}--\eqref{Cb-est} that $u=f_u$. 

This proves
\begin{prop}
  \label{reconstruction-prop}
Let $u\in B^{\fracpi}_{p,q}(\Rn)$ for some $p\in[1,\infty]$
and $q\le1$. Then the function $f_u(t)$ given by \eqref{fu-eq}--\eqref{Cb-eq}
defines a distribution $\Lambda_{f_u(t)}$, by Proposition~\ref{inj-prop}, that
coincides with $u$; that is, $\Lambda_{f_u(t)}=u$. 
\end{prop}

Thereby \eqref{1.13} has been verified for the result in
Theorem~\ref{obord-thm}, so the distribution trace $u(0)$ is defined for every
$u\in B^{\fracpi}_{p,1}$; viewing $u$ as an element of $C(\R,\cal D'(\Rn))$
gives $u(0)=f_u(0)=\g u$ as desired. 
In particular $\g u$ in \eqref{1.1} is independent of the
choice of partition of unity.

\section{Surjectiveness}
Since $\g$ in \eqref{2.7} has dense range, it is for $q=1$ surjective precisely
when its adjoint $\g^*$ has a bounded inverse from $\op{ran}(\g^*)$ to
$L_p^*$ (see e.g.\ \cite[Th.~4.15]{R}).

For $1\le p<\infty$ and $q=1$ the adjoint is bounded, when $p+p'=pp'$,
\begin{equation}
  \g^*\colon L_{p'}(\R^{n-1})\to B^{\frac1{p'}-1}_{p',\infty}(\Rn);
  \label{2.8}
\end{equation}
cf.\ \cite{T2} for the dual space; and $\g^*
u=u\otimes\delta_0$ for $u\in L_{p'}$ since for $\varphi\in \cal S$
\begin{equation}
  \dual{\g^* u}{\varphi}=\dual{u}{\varphi(\cdot,0)}
  =\dual{u\otimes\delta_0}{\varphi}.
  \label{2.9}
\end{equation}
It remains to be shown, with primes omitted for simplicity, that
\begin{equation}
  \norm{u}{L_p}\le c\norm{u\otimes\delta_0}{B^{\fracpi-1}_{p,\infty}}
  =: c\cdot B(u)
  \label{2.10} 
\end{equation}
for all $u\in L_p(\R^{n-1})$ whenever $p\in\,]1,\infty]$.

Using Lemma~\ref{Liz-lem} we have a
partition of unity $1=\sum_{k=0}^{\infty}\sum_{J\neq\emptyset} \Theta_{J,k}$,
where each $\Theta_{J,k}$ is a product:
\begin{equation}
  \begin{gathered}
  \Theta_{J,k}(\xi)=\eta_{J,k}(\xi')\theta_{J,k}(\xi_n),  \\
  \eta_{J,k}(\xi')=\eta_J(2^{-k}\xi'),\quad 
  \theta_{J,k}(\xi_n)=\theta_J(2^{-k}\xi_n)\quad\text{for $k>0$}.
  \end{gathered}
  \label{2.11}
\end{equation}
By \eqref{2.8}, the corresponding $B^{\fracpi-1}_{p,\infty}$-norm with
supremum over $(J,k)$ gives
\begin{equation}
  \begin{split}
  \norm{\check\theta_J}{L_p}\norm{\check\eta_{J,k}*u}{L_p}  
  &= 2^{j(\fracpi-1)}\Norm{\cal F^{-1}(\Theta_{J,k}\cal
   F(u\otimes\delta_0))}{L_p}
\\
  &\le B(u)<\infty.
  \label{2.12}
\end{split}
\end{equation}
Since $\eta_J(0)\ne0$ for some $J$, we can take $J$ such that 
\begin{equation}
  \check\eta_{J,k}*u\to a\cdot u\quad\text{in $\cal D'$ for
  $k\to\infty$}
  \label{2.14}
\end{equation}
if $a:=\int\check\eta_J\ne0$. The w$^*$-compactness of the balls in
$L_p$ together with \eqref{2.14}--\eqref{2.12} show that \eqref{2.10}
holds with $c$ equal to $(a\norm{\check\theta_J}{L_p})^{-1}$.

From the Besov spaces' point of view the surjectiveness is proved in a
natural way above; essentially it is known from the technical report
\cite{Pee75}.

\bigskip

For $q\le1$ the dual of $B^{\fracpi}_{p,q}$ is independent of $q$, because
$(B^{\fracpi}_{p,q})^*=B^{-\fracpi}_{p',\infty}$ then; cf.\
\cite[2.11.2]{T2}. Therefore the adjoint remains equal to \eqref{2.8} 
for $q<1$, so it suffices to show that the Closed Range Theorem is valid when
the domain is a quasi-Banach space. 

Observe first, for precision, that $B^{s}_{p,q}$ is an F-space in Rudin's
terminology  \cite{R} when $d(u,v):=\norm{u-v}{B^{s}_{p,q}}^\lambda$ and
$\lambda=\min(1,p,q)$. Hence continuity and boundedness are equivalent for
operators between these quasi-Banach spaces \cite[1.32]{R}. 

Moreover, defining the operator norm in the usual way, $\BBb B(X,Y)$ becomes a
quasi-Banach space; $\|S+T\|\le c(\|S\|+\|T\|)$ holds with the same constant
as it does for $\norm{\cdot}{Y}$. In particular, $X^*$ is always a Banach
space. As usual each $T\in\BBb B(X,Y)$ has an adjoint $T^*\in\BBb B(Y^*,X^*)$.

\begin{prop}
  \label{cr-thm}
Let $X$ be a quasi-Banach space such that $\norm{\cdot}{X}^\lambda$ is
subadditive for some $\lambda\in\,]0,1]$, let $Y$ be a Banach space and
$T\colon X\to Y$ be a bounded linear operator.
When $\overline{T(X)}=Y$, then boundedness of $T^{*-1}$ from $T^*(Y^*)$ to
$Y^*$ implies that $T$ is surjective, i.e.\ $T(X)=Y$.
\end{prop}

\begin{proof}
Since $\ker T^*\subset T(X)^\perp =\{0\}$ the inverse is well defined; by
assumption there is a constant $c<\infty$ such that 
\begin{equation}
  \norm{y^*}{Y^*}\le c\norm{T^*y^*}{X^*} \quad\text{for all}\quad y^*\in Y^*.
  \label{*bd-eq}
\end{equation}
This inequality implies that $T$ is open. Indeed, if $X$ is a
Banach space this is the content of \cite[Lem.~4.13]{R}. When only $Y$ is
assumed to be a Banach space, the reduction from part (b) to (a) there carries
over verbatim (since the Hahn--Banach theorem is only used for $Y$), and in
the proof of (a) the sequence $(\varepsilon_n)$ should be picked in
$\ell_\lambda$ such that 
$\sum_{n=1}^\infty \varepsilon_n^\lambda < 1-\norm{y_1}{Y}^\lambda$.
Then the sequences $(x_n)$ and $(y_n)$ defined there satisfy
\begin{equation}
  \sum_{n=1}^\infty\norm{x_n}{X}^\lambda\le
  \norm{x_1}{X}^\lambda +\sum_{n=1}^\infty \varepsilon_n^\lambda <
  \norm{y_1}{Y}^\lambda+(1-\norm{y_1}{Y}^\lambda)=1;
\end{equation}
hence $x=\sum x_n$ converges in $X$ and has $\norm{x}{X}<1$ as desired. 

Thus \eqref{*bd-eq} implies that $T$ is an open mapping, 
but as such it's necessarily
surjective.  
\end{proof}

Altogether this shows that $L_p(\R^{n-1})$ is the image of
$B^{\fracpi}_{p,q}(\Rn)$ under $\g$ for every $q\le 1$ when $1\le p<\infty$.
 
\begin{rem}
  It is known that every quasi-Banach space $X$ has an equivalent quasi-norm
  such that $\norm{\cdot}{X}^{\lambda}$ is sub-additive for some
  $\lambda\in\,]0,1]$. In view of this, the proposition holds for all
  quasi-Banach spaces.  
\end{rem}

\section{The borderline for $0<p<1$}   \label{nbord-sect}
Since the boundedness in Theorem~\ref{nbord-thm} is proved in
Section~\ref{boundedness-sect} above, it remains to show the claim on the
integral-exponents.

That it is necessary for $p<q\le1$ in Theorem~\ref{nbord-thm} to let
$B^{(n-1)(\fracci1q-1)}_{q,\infty}(\R^{n-1})$ receive $\g u$
follows because the inequality $r\ge q$ is implied by the estimate
\begin{equation}
  \norm{\g u}{B^t_{r,\infty}}\le c\norm{u}{B^{\fracci np-n+1}_{p,q}}.
  \label{g0F0-cnd} 
\end{equation}
To show this implication, it suffices to extend the $\omega_N$ in the proof of
Proposition~\ref{qp-prop} by taking some $\eta\in\cal S(\R)$ satisfying
$\supp \eta\subset \,]1,2[\,$ and $\check\eta(0)=1$ and set
\begin{equation}
  E\omega_N(x)=\sum_{k=1}^N \check\eta(2^k x_n) \check\Psi_k(x'-x'_k).
\end{equation}
Using \eqref{g0F0-cnd} and part (ii) of Theorem~\ref{Y-thm} to estimate the
Besov norm of $E\omega_N$, it is easily seen that
\begin{equation}
  \norm{\omega_N}{B^t_{r,\infty}(\R^{n-1})}=
  \norm{\g E\omega_N}{B^t_{r,\infty}}\le
  \norm{E\omega_N}{B^{\fracci np-n+1}_{p,q}(\Rn)}\le
  cN^{\fracci1q}.
\end{equation}
Because of \eqref{1r-eq} the inequality $r\ge q$ holds.

\section{Final remarks}   \label{frem-sect}
\begin{rem}
  \label{cx1-rem}
In addition to \eqref{i8},
$\psi_k:=2^{k(n-1)}\eta(2^k\cdot)\varphi\to \delta\otimes\varphi$ 
in $B^s_{p,1}$ for
$k\to\infty$ when $\fracp<s<\fracnp-n$ (which entails $p<1-\fracc
1n$), at least if $\hat\eta=1$ in a ball around $\xi'=0$. For
by Remark~\ref{Y-rem}, $\delta_0=\eta+\sum_{k=1}^\infty
(2^{k(n-1)}\eta(2^k\cdot)-2^{(k-1)(n-1)}\eta(2^{k-1}\cdot))$ 
converges in $B^{s}_{p,1}(\R^{n-1})$
while $\cdot\otimes\varphi$ maps
continuously into $B^{s}_{p,1}(\Rn)$ by \cite{F3}. Hence
$\psi_k\to\delta_0\otimes\varphi$ there, and $T\psi_k\to 0$ as shown
in \eqref{i8}; i.e.~\eqref{i5} holds. 
\end{rem}

\begin{rem}
  \label{cx2-rem}
For $s=\fracnp-n+1$ it is useful to consider 
\begin{equation}
  v_k(x)=\tfrac{1}{k}\sum_{l=k+1}^{2k}
2^{l(n-1)}f(2^lx')g(2^lx_n)
\end{equation}
for Schwartz functions $f$ and $g$ with their spectra in balls of
radius $1/2$ such that $\int f=1$ and $g(0)=1$. As shown in
\cite[Lem.~2.8]{JJ96ell}, $\g  v_k\to \delta_0$ in $\cal D'$
while $v_k\to 0$ in $B^{\fracci np-n+1}_{p,q}$ if $q>1$, so that $\g$
is only continuous from $B^{\fracci np-n+1}_{p,q}$ if $q\le1$.

However, $Tv_k= \frac{1}{k}\sum_{l=k+1}^{2k} 2^{l(n-1)}f(2^l\cdot)$
since $v_k\in\cal S$, and its norm  $\norm{T v_k}{L_p}$ is 
$\cal O(k^{(n-1)(1-\fracpi)})$ 
and so tends to $0$ for $k\to\infty$; that is, already at the
borderline $\g$ and $T$ behave differently. 
\end{rem}

\begin{rem}   \label{ell-rem}
For $\Omega$ equal to the unit ball in $\Rn$, $n\ge3$, Franke and
Runst \cite[Sect.~6.5]{FR95} proved that $B^{\fracci
np-n+1}_{p,\infty}(\overline{\Omega})$ contains an
infinite-dimensional solution space for the problem 
\begin{equation}
  \mlap u=0 \quad\text{in $\Omega$},\quad Tu=0\quad\text{on $S^{n-1}$}.
  \label{dir-pb}  
\end{equation}
In fact, for each boundary point $z\in S^{n-1}$ they showed that
$\Phi(x-z)-\tfrac{1}{n-2}z\cdot\grad\Phi(x-z)$, where
$\Phi(x)=c|x|^{2-n}$ is the fundamental solution of $\mlap$, belongs
to this space and solves \eqref{dir-pb}.

Moreover, in \cite{JJ93,JJ96ell} it was proved that the Boutet de~Monvel
calculus of pseudo-differential boundary operators (for elliptic
problems) extends nicely to spaces with $p<1$. However for trace
operators and $P_\Omega+G$ that precisely \emph{have} class $r\in\Z$, it was
proved that $s\ge\fracnp-n+r$ is \emph{necessary} for continuity from
$B^{s}_{p,q}$ to $\cal D'$ when $p<1$. 

Taken together, these facts show that not only the usual Fredholm properties
but also the continuity of solution operators for elliptic problems break down
for $p<1$ unless $s=\fracnp-n+r$ is taken as the borderline for
operators of class $r$. (For the Dirichl\'et realisation of
$\mlap$, the latter fact was also shown by Chang, Krantz and Stein 
\cite{ChaKraSte93}.)
\end{rem}

%
%  REFERENCES
%
\providecommand{\bysame}{\leavevmode\hbox to3em{\hrulefill}\thinspace}

\end{document}